\documentclass[a4paper,11pt]{amsart}
\usepackage[english]{babel}
\usepackage{amsmath}
\usepackage{amssymb}
\usepackage{mathrsfs}
\usepackage{enumerate}
\usepackage{mathtools}
\usepackage{tikz,tkz-euclide}
\usetikzlibrary{shapes.geometric, arrows}
\tikzstyle{startstop} = [rectangle, rounded corners, minimum width=2cm, minimum height=1cm,text centered, draw=black, fill=red!30]
\tikzstyle{io} = [trapezium, trapezium left angle=70, trapezium right angle=110, minimum width=3cm, minimum height=1cm, text centered, draw=black, fill=blue!30]
\tikzstyle{process} = [rectangle, minimum width=3cm, minimum height=1cm, text centered, draw=black, fill=orange!30]
\tikzstyle{decision} = [diamond, minimum width=3cm, minimum height=1cm, text centered, draw=black, fill=green!30]
\tikzstyle{arrow} = [thick,->,>=stealth]

\usepackage{pgfplots}
\usepackage{hyperref}
\usetikzlibrary{arrows}
\usepackage{makecell}

\usepackage[most]{tcolorbox}
\usepackage{lmodern} 

\usepackage[]{algorithm2e}

\usepackage[most]{tcolorbox}

\newtheorem{thm}{Theorem}[section]
\newtheorem{Lemma}[thm]{Lemma}
\newtheorem{Proposition}[thm]{Proposition}
\newtheorem{Corollary}[thm]{Corollary}
\newtheorem{Conjecture}[thm]{Conjecture}
\newtheorem*{thm*}{Theorem}

\theoremstyle{definition}
\newtheorem{Notation}[thm]{Notation}

\newtheorem{Definition}[thm]{Definition}
\newtheorem{Remark}[thm]{Remark}

\newtheorem{Example}[thm]{Example}

\definecolor{wwwwww}{rgb}{0.4,0.4,0.4}

\newcommand{\G}{\mathbb{G}}

\newcommand{\p}{\mathbb{P}}

\DeclareMathOperator{\Sing}{Sing}
\DeclareMathOperator{\expdim}{expdim}

\DeclareMathOperator{\Sec}{Sec}

\renewcommand{\sec}{\mathbb{S}ec}

\hypersetup{pdfpagemode=UseNone}
\hypersetup{pdfstartview=FitH}
		
\begin{document}

\title{Ample bodies and Terracini loci of projective varieties}

\author[Antonio Laface]{Antonio Laface}
\address{\sc Antonio Laface\\
Departamento de Matematica, Universidad de Concepci\'on\\
Casilla 160-C, Concepci\'on\\
Chile}
\email{alaface@udec.cl}

\author[Alex Massarenti]{Alex Massarenti}
\address{\sc Alex Massarenti\\ Dipartimento di Matematica e Informatica, Universit\`a di Ferrara, Via Machiavelli 30, 44121 Ferrara, Italy}
\email{alex.massarenti@unife.it}
%

\date{\today}

\begin{abstract}
We introduce the notion of ample body of a projective variety and use it to prove emptiness results for Terracini loci and specific identifiability results for toric and homogeneous varieties.
\end{abstract}

\maketitle
\setcounter{tocdepth}{1}
\tableofcontents

\section{Introduction}
Let $X\subset\mathbb{P}^N$ be an irreducible projective variety. The $h$-Terracini locus of $X$ parametrizes unordered $h$-uples of distinct points $x_1,\dots,x_h \in X$ at which the tangent spaces span a linear space of dimension smaller than expected. Terracini loci have been introduced in \cite{BC21} and then studied for several relevant varieties such as Veronese and Segre varieties \cite{BBS20}, \cite{Bal22a}, \cite{Bal22}, \cite{CG23}, \cite{BV23}.

These loci are closely related to the concepts of secant defectiveness and identifiability. The $h$-secant variety
$\sec_h(X)\subset\mathbb{P}^N$ of $X$ is the Zariski closure of the
union of the $(h-1)$-planes spanned by collections of $h$ points of
$X$.

The expected dimension of $\mathbb{S}ec_{h}(X)$ is $\expdim(\mathbb{S}ec_{h}(X)):= \min\{nh+h-1,N\}$. The actual dimension of $\mathbb{S}ec_{h}(X)$ may be  smaller than the expected
one. The variety $X$ is $h$-defective if
$\dim(\mathbb{S}ec_{h}(X)) 
< \expdim(\mathbb{S}ec_{h}(X))$ and 
$h$-identifiable if through a general point of $\sec_h(X)$ there passes a unique
$(h-1)$-plane spanned by $h$ points of $X$. Furthermore, if this last property holds for a special point $p\in \sec_h(X)$ we say that $p$ is $h$-identifiable. 

The Terracini's lemma yields that $X$ is $h$-defective if and only if the $h$-Terracini locus of $X$ coincides with the symmetric product $X^h/S_h$. In Section \ref{TSIB} we relate the emptiness of certain Terracini loci to specific identifiability. This property, especially when the ambient projective space parametrized tensors, is relevant also in applied sciences, for instance in psycho-metrics, chemo-metrics, signal processing, numerical linear algebra, computer vision, numerical analysis, neuroscience and graph analysis \cite{BK09}.

In Section \ref{TVAB} we associate to a projective variety $X$ a geometric object $\mathscr A_X$ that we call the ample body of $X$. This is the convex hull of ample divisor classes of $X$.
When the Mori cone of $X$ is rational polyhedral $\mathscr A_X$ turns out to be a polyhedron which is the Minkowski sum of a rational polytope $A_X$ and of the nef cone of $X$.
In many case we manage to control the geometry of $\mathscr A_X$ and to use it to prove emptiness results for Terracini loci. For instance, summing up Proposition \ref{TC2} and Theorem \ref{main_to}, for toric varieties we have the following:

\begin{thm}\label{thA}
Let $P\subseteq M_{\mathbb Q}$ be a full 
dimensional lattice polytope such that 
the corresponding projective toric variety $X_P\subseteq\mathbb P^{|P\cap M|-1}$, embedded by a complete linear system $|L|$, is smooth, and set
$$
 \ell(P) := \min\{|L\cap M|-1\, | \, L \text{ is a 1-dimensional face of }P\}.
$$ 
For $2$-Terracini loci the following are equivalent:
\begin{itemize}
\item[-] $\ell(P)\geq 3$;
\item[-] $T_2(X)$ is empty;
\item[-] $X_P$ does not contain conics.
\end{itemize}
Furthermore, if $A_{X_P}$ is a normal lattice polytope then the following are equivalent:
\begin{itemize}
\item[-] $\ell(P) = s$ with $s\geq 2h-1$;
\item[-] $L\in s\cdot\mathscr A_X$;
\item[-] $T_h(X)$ is empty.
\end{itemize}
 
\end{thm}
In Section \ref{appli} we apply Theorem \ref{thA} and the theory of ample bodies to specific toric varieties such as Segre-Veronese varieties, toric varieties of Picard rank two, and to homogeneous varieties. As a sample we summarize our main results for Segre-Veronese varieties in Corollaries \ref{TCSV}, \ref{TCSV2}.  

Let $\pmb{n}=(n_1,\dots,n_r)$ and $\pmb{d} = (d_1,\dots,d_r)$ be two $r$-uples of positive integers, with $n_1\leq \dots \leq n_r$ and
$N(\pmb{n},\pmb{d})=\prod_{i=1}^r\binom{n_i+d_i}{n_i}-1$. The Segre-Veronese variety $SV_{\pmb{d}}^{\pmb{n}}$ is the image in
$\mathbb{P}^{N(\pmb{n},\pmb{d})}$ of $\mathbb{P}^{n_1}\times\dots\times\mathbb{P}^{n_r}$ via the embedding induced by $ 
\mathcal{O}_{\mathbb{P}^{\pmb{n}} }(d_1,\dots,d_r)=\mathcal{O}_{\mathbb{P}(V_1^{*})}(d_1)\boxtimes\dots\boxtimes \mathcal{O}_{\mathbb{P}(V_1^{*})}(d_r)$.

\begin{thm}\label{thB}
If $h\leq\lceil\frac{d_i}{2}\rceil$ for all $i = 1,\dots,r$ then $T_h(SV_{\pmb{d}}^{\pmb{n}})$ is empty. Furthermore, if $2h\leq\lceil\frac{d_i}{2}\rceil$ for all $i = 1,\dots,r$ then any point of $\sec_h(SV_{\pmb{d}}^{\pmb{n}})\setminus \sec_{h-1}(SV_{\pmb{d}}^{\pmb{n}})$ is $h$-identifiable and $\sec_h(SV_{\pmb{d}}^{\pmb{n}})$ is smooth outside of $\sec_{h-1}(SV_{\pmb{d}}^{\pmb{n}})$.
\end{thm}

Finally, in Proposition \ref{grass} we describe $2$-Terracini loci of Grassmannians, and in Section \ref{Fano} we discuss ample bodies of toric Fano varieties and prove the following result:

\begin{thm}\label{thC}
Let $P\subseteq M_{\mathbb Q}$ be a full dimensional lattice polytope such that 
the corresponding projective toric variety $X_P\subseteq\mathbb P^{|P\cap M|-1}$, embedded by a complete linear system $|L|$, is smooth and Fano of dimension at most four. Then $A_{X_P}$ is a lattice point, and $T_h(X)$ is empty if and only if  $\ell(P) = s$ with $s\geq 2h-1$.
\end{thm}

The proof of Theorem \ref{thC} is based on our results for toric varieties, the data contained in the \href{http://www.grdb.co.uk/search/toricsmooth}{Graded Ring Database}, and a Magma \cite{Magma97} script we developed to compute $A_X$.

\subsection*{Acknowledgments}
This collaboration started during the \href{https://agates.mimuw.edu.pl/index.php/agates/geometry-of-secants}{"\textit{Geometry of secants workshop}"} held at \href{https://www.impan.pl/en}{IMPAN} from October $24$ to October $28$, $2022$. We thank all the organizers for the stimulating environment, Luca Chiantini and Ciro Ciliberto for introducing us to Terracini loci during the workshop, and Rick Rischter for many helpful conversations. 

This work is partially supported by the Thematic Research Programme "\textit{Tensors: geometry, complexity and quantum entanglement}", University of Warsaw, Excellence Initiative - Research University and the Simons Foundation Award No. 663281 granted to the Institute of Mathematics of the Polish Academy of Sciences for the years 2021-2023. The first author has been partially supported by Proyecto FONDECYT Regular n. 1230287.

\section{Terracini loci, specific identifiability and Bronowski's conjecture}\label{TSIB}
Let $X\subset\mathbb{P}^N$ be an irreducible and non-degenerate 
variety of dimension $n$ and let $\Gamma_h(X)\subset X\times \dots
\times X\times\G(h-1,N)$, where $h\leq N$, be the closure of the graph
of the rational map $\alpha: X\times\dots\times X \dasharrow
\G(h-1,N)$ taking $h$ general points to their linear span. Observe that $\Gamma_h(X)$ is irreducible
and reduced of dimension $hn$. Let $\pi_2:\Gamma_h(X)\to\G(h-1,N)$ be
the natural projection, and
$\mathcal{S}_h(X):=\pi_2(\Gamma_h(X))\subset\G(h-1,N)$. Again
$\mathcal{S}_h(X)$ is irreducible and reduced of dimension
$hn$. Finally, consider
$$\mathcal{I}_h=\{(x,\Lambda) \: | \: x\in \Lambda\}\subset\mathbb{P}^N\times\G(h-1,N)$$
with the projections $\pi_h^X$ and $\psi_h^X$ onto the factors. The abstract $h$-secant variety is the irreducible variety
$$\Sec_{h}(X):=(\psi_h^X)^{-1}(\mathcal{S}_h(X))\subset \mathcal{I}_h.$$
The $h$-secant variety is defined as
$$\sec_{h}(X):=\pi_h^X(\Sec_{h}(X))\subset\mathbb{P}^N.$$
It immediately follows that $\Sec_{h}(X)$ is an $(hn+h-1)$-dimensional
variety with a $\mathbb{P}^{h-1}$-bundle structure over
$\mathcal{S}_h(X)$. We say that $X$ is $h$-defective if
$\dim\sec_{h}(X)<\min\{\dim\Sec_{h}(X),N\}$. We will denote by $\sec_h(X)^{o}$ the union of the $(h-1)$-planes spanned by $h$ linearly independent points of $X$.

\begin{Definition}
When $\pi_h^X:\Sec_h(X)\rightarrow\sec_h(X)$ is generically finite we will call
its degree the $h$-secant degree of $X$, and we will say that $X$ is
$h$-identifiable when its $h$-secant degree is one.  
\end{Definition}

\begin{Remark}\label{SI}
Note that even when $\pi_h^X:\Sec_h(X)\rightarrow\sec_h(X)$ is birational it might have positive dimensional fibers or zero dimensional fibers of degree bigger than one. In this last case $\sec_h(X)$ will not be normal. We will say that $p\in \sec_h(X)$ is $h$-identifiable if $(\pi_h^{X})^{-1}(p)$ has degree one. 
\end{Remark}

\begin{Definition}\label{TC_Def}
Let $X\subset\mathbb{P}^N$ be a smooth, irreducible and non-degenerate variety of dimension $n$. The $h$-th Terracini locus $T_h(X)$ of $X$ is the closure of
$$
T_h(X)^{o} = \{\{x_1,\dots,x_h\} \: | \: x_i\neq x_j \: \text{and} \: \dim(\left\langle T_{x_1}X,\dots,T_{x_h}X\right\rangle) < \min\{hn+h-1,N\}\}\subseteq X^h/S_h
$$
in the $h$-th symmetric product $X^h/S_h$ of $X$.
\end{Definition}

\begin{Definition}\label{TC_Mat}
Given a smooth, irreducible and non-degenerate variety $X\subset\mathbb{P}^N$ of dimension $n$ and a local parametrization 
$$
\begin{array}{lccc}
\phi: & k^n & \longrightarrow & X\subset\mathbb{P}^{N}\\
& (u_1,\dots,u_n) & \mapsto & [1,\phi_1(u_1,\dots,u_n),\dots,\phi_N(u_1,\dots,u_n)]
\end{array}
$$
we define the Terracini matrix $T_{X}(x_1,\dots,x_h)$ of the $h$ points $x_1,\dots,x_h\in \phi(k^n)$ as
$$
T_{X}(x_1,\dots,x_h) = 
\left(\begin{array}{cccc}
1 & \phi_1(p_1) & \dots & \phi_N(p_1)\\ 
\vdots & \vdots & \ddots & \vdots\\
1 & \phi_1(p_h) & \dots & \phi_N(p_h)\\
0 & \phi_{1,u_1}(p_1) & \dots & \phi_{N,u_1}(p_1)\\
\vdots & \vdots & \ddots & \vdots\\
0 & \phi_{1,u_n}(p_h) & \dots & \phi_{N,u_n}(p_h)
\end{array} 
\right)
$$
where $p_i = \phi^{-1}(x_i)$ and $\phi_{i,u_j} = \frac{\partial \phi_i}{\partial u_j}$.
\end{Definition}

In the following we relate the emptiness of Terracini loci to specific identifiability. 

\begin{Proposition}\label{nid}
If $T_h(X)$ is empty then for any $p\in \sec_h(X)^{o}\setminus \sec_{h-1}(X)^{o}$ the fiber $(\pi_h^X)^{-1}(p)$ consists of a finite number of points. 
\end{Proposition}
\begin{proof}
Let $\left\langle x_1,\dots,x_h\right\rangle$ be an $h$-plane containing $p$ with $x_1,\dots,x_h\in X$. Assume that $(\pi_h^X)^{-1}(p)$ has positive dimension. Then the image of the differential of $\pi_h^X$ at $p$ has dimension smaller that $\dim(X)h+h-1$. Since $T_{x_1}X,\dots,T_{x_h}X$ are contained in the image of the differential of $\pi_h^X$ at $p$ we conclude that $\{x_1\dots,x_h\}\in T_h(X)$. 
\end{proof}

\begin{Corollary}\label{Col1}
If $T_h(X)$ is empty, $X$ is $h$-identifiable and $\sec_h(X)^{o}\setminus \sec_{h-1}(X)^{o}$ is normal then any point $p\in \sec_h(X)^{o}\setminus \sec_{h-1}(X)^{o}$ is $h$-identifiable.
\end{Corollary}
\begin{proof}
Since $X$ is $h$-identifiable $\pi_h^X$ is birational, and since $T_h(X)$ is empty Proposition \ref{nid} yields that $\pi_h^X$ is finite over $\sec_h(X)^{o}\setminus \sec_{h-1}(X)^{o}$. To conclude it is enough to note that if the fiber $(\pi_h^X)^{-1}(p)$ over $p \in\sec_h(X)^{o}\setminus \sec_{h-1}(X)^{o}$ has degree bigger that one then $\sec_h(X)^{o}$ is not normal at $p$. 
\end{proof}

\begin{Corollary}
Let $X\subset\mathbb{P}^N$ be a smooth projective variety of dimension $n$ embedded by a line bundle $L = \omega_X\otimes A^{2(n+1)}\otimes B$, where $A$ is very ample and $B$ is nef. If $T_2(X)$ is empty and $X$ is $2$-identifiable then any point $p\in \sec_2(X)\setminus X$ is $2$-identifiable.
\end{Corollary}
\begin{proof}
By \cite[Corollary C]{Ull16} $\sec_2(X)$ is normal. Hence, Corollary \ref{Col1} yields that any point $p\in \sec_2(X)^{o}\setminus X^{o}$ is $2$-identifiable. Now, let $p\in \sec_2(X)\setminus X$ be a point lying on a line $L$ tangent to $X$ at $x\in X$. Assume that the is another line $L'$ through $p$ that is secant to $X$. If $L'$ is tangent to $X$ at a point $x'$ then $T_xX\cap T_{x'}X\neq\emptyset$ and hence $T_2(X)\neq\emptyset$. If $L'$ is a proper secant then arguing as in the proof of Proposition \ref{nid} we get that the fiber $(\pi_2^X)^{-1}(p)$ is finite and hence $\sec_2(X)$ is not normal at $p$ contradicting \cite[Corollary C]{Ull16}. 
\end{proof}

\begin{Proposition}\label{h-2h}
Let $X\subset\mathbb{P}^N$ be a smooth projective variety of dimension $n$. If $T_{2h}(X)$ is empty then any point of $\sec_h(X)\setminus \sec_{h-1}(X)$ is $h$-identifiable. Furthermore, $\Sing(\sec_h(X))\subset \sec_{h-1}(X)$.
\end{Proposition}
\begin{proof}
Let $p\in  \sec_h(X)\setminus \sec_{h-1}(X)$ and assume that there are two $(h-1)$-planes $H,H'$ through $p$ intersecting $X$ in schemes of dimension zero and degree $h$ supported respectively on $\{x_1,\dots,x_a\}$ and $\{x_1',\dots,x_b'\}$.

Then $H\subset\left\langle T_{x_1}X,\dots,T_{x_a}X\right\rangle$ and $H'\subset\left\langle T_{x_1'}X,\dots,T_{x_a'}X\right\rangle$ yield that $p\in \left\langle T_{x_1}X,\dots,T_{x_a}X\right\rangle \cap \left\langle T_{x_1'}X,\dots,T_{x_a'}X\right\rangle$ and hence $T_{a+b}(X)\neq \emptyset$. To conclude it is enough to note that $a+b\leq 2h$ and that $T_{2h}(X) = \emptyset$ implies  $T_{h'}(X) = \emptyset$ for all $h' \leq 2h$.

Therefore, $\pi_h^X$ is $1$-to-$1$ over $\sec_h(X)\setminus \sec_{h-1}(X)$ and since $T_h(X) = \emptyset$, arguing as in Proposition \ref{nid}, we have that $\pi_h^X$ is a submersion over $\sec_h(X)\setminus \sec_{h-1}(X)$ and hence $\sec_h(X)\setminus \sec_{h-1}(X)$ is smooth. 
\end{proof}

Thanks to the theory developed in \cite[Section 2]{MM22} it is possible to relate Terracini loci to the Bronowski's conjecture \cite[Remark 4.6]{CR06} which has been proved false and amended, by requiring the non degeneracy of the Gauss map of $X$, in \cite[Theorem 1.3, Conjecture 1.4]{MM22}.

\begin{Notation}
Let $X\subset\p^N$ be an irreducible and non-degenerate variety. A general $h$-tangential projection of $X$ is a linear projection $\tau_{x_1,\dots,x_h}^X:X\dasharrow X_{h}\subset\mathbb{P}^{N_{h}}$ from the linear span of $h$ tangent
spaces $\left\langle T_{x_1}X,\dots, T_{x_h}X\right\rangle$ where
$x_1,\dots,x_h\in X$ are general points. When there will be no danger of confusion we will denote a general $h$-tangential projection $\tau_{x_1,\dots,x_h}^X$ simply by $\tau_{h}^X$.
\end{Notation}

\begin{Notation}\label{TR}
Consider the map
$\pi_{h+1}^X:\Sec_{h+1}(X)\rightarrow\sec_{h+1}(X)\subset\mathbb{P}^N$. For
a general point $p\in\sec_{h}(X)\subset\sec_{h+1}(X)$ we split the fiber
$\pi_{h+1}^X(p)$ in two parts $T_p^{h},R_p^h$ defined as
follows:  
\begin{itemize}
\item[-] the general point of  the trivial part $T_p^{h}$ is a pair
  $(p,\Lambda)$ where $\Lambda$ is an $h$-plane $(h+1)$-secant to $X$
  of the form $\Lambda = \left\langle x,\Lambda'\right\rangle$ with $x
  \in X$ and $\Lambda'$ an $(h-1)$-plane $h$-secant to $X$ and passing
  through $p$; 
\item[-] the residual part $R_p^{h}$ is the closure of the complement of $T_p$ in $(\pi_{h+1}^X)^{-1}(p)$. 
\end{itemize}
\end{Notation}

The following is the revised version of Bronowski's conjecture in \cite[Conjecture 1.4]{MM22}.

\begin{Conjecture}\label{con:Bro_mio}
Let $X\subset\mathbb{P}^{N}$ be an irreducible and non-degenerate variety with non degenerate Gauss map. The variety $X$ is $h$-identifiable if and only if a general $(h-1)$-tangential projection $\tau_{h-1}^X:X\dasharrow X_{h-1}\subset\mathbb{P}^{N_{h-1}}$ is birational. 
\end{Conjecture}

\begin{Proposition}\label{P_Bro}
Let $X\subset\mathbb{P}^{N}$ be an irreducible and non-degenerate variety such that $T_{2h-1}(X)$ is empty. Then $X$ has non degenerate Gauss map and Conjecture \ref{con:Bro_mio} holds true for $X$ and the integer $h$, that is $X$ is $h$-identifiable if and only if a general $(h-1)$-tangential projection $\tau_{h-1}^X:X\dasharrow X_{h-1}\subset\mathbb{P}^{N_{h-1}}$ is birational.
\end{Proposition}
\begin{proof}
First, note that the Gauss map of $X$ is degenerate if and only if a general tangent space of $X$ is tangent to $X$ along a positive dimensional subvariety and in this case $T_{2}(X)$ is non empty. 

Let $p\in \sec_{h-1}(X)$ be a general point, and consider an $(h-2)$-plane $\left\langle x_1,\dots,x_{h-1}\right\rangle$ through $p$ with $x_1,\dots,x_{h-1}\in X$. Assume that the residual part $R_p^{h-1}$ is non empty. Then there is an $(h-1)$-plane $\left\langle x_1',\dots,x_{h}'\right\rangle$ through $p$ with $x_1',\dots,x_{h}'\in X$ and at least two of the $x_i'$ do not belong to $\{x_1,\dots,x_{h-1}\}$. Therefore, $T_{2h-1}(X)\neq \emptyset$ contradicting the hypotheses. Therefore, $R_p^{h-1}$ is empty and to conclude it is enough to apply \cite[Corollary 2.18]{MM22}. 
\end{proof}

\section{Ample bodies and toric varieties}\label{TVAB}
Let $N$ be a rank $n$ free abelian group, $M := {\rm Hom}(N,\mathbb Z)$ its dual and $M_{\mathbb Q} := M\otimes_{\mathbb Z}\mathbb Q$ the corresponding rational vector space. Let $P\subseteq M_{\mathbb Q}$ be a full-dimensional lattice polytope, that is the convex hull of finitely many points in $M$ which do not lie on a hyperplane. The polytope $P$ defines a polarized pair $(X_P,H)$ consisting of the toric variety $X_P$ together with a very ample Cartier
divisor $H$ of $X_P$. More precisely $X_P$ is the Zariski closure of the image of
the monomial map
\stepcounter{thm}
\begin{equation}\label{par}
\begin{array}{lccc}
\phi_P: & (k^*)^n & \longrightarrow & \mathbb{P}^{N}\\
& u & \mapsto & [\chi^m(u)\, :\, m\in P\cap M]
\end{array}
\end{equation}
where $P\cap M = \{m_0,\dots,m_N\}$, $\chi^m(u)$ denotes the Laurent monomial in the variables $(u_1,\dots,u_n)$ defined by the point $m$, and $H$ is a hyperplane section of $X_P$.

\begin{Lemma}\label{spec}
Let $X$ be a smooth projective toric
variety with one-parameter subgroup 
lattice $N$ and let $p_1,p_2\in X$ 
be two distinct points.
Then there exists $v\in N$ such 
that the two limits
\[
 q_1 := \lim_{t\to 0}t^v\cdot p_1,
 \qquad
 q_2 := \lim_{t\to 0}t^v\cdot p_2
\]
are two distinct points which are either 
torus invariant or lie on a 
torus invariant curve.
\end{Lemma}
\begin{proof}
Let $T$ be the big torus of $X$ and 
let $X_i$ be the Zariski closure of
the orbit $T\cdot p_i$. Recall that 
both $q_1$ and $q_2$ are invariant 
fixed points for a general $v\in N$.
If $X_1\cap X_2$ is empty, then for 
such a general $v\in N$ we get two
distinct invariant points.
Assume now that $X_1\cap X_2$ is 
non empty. Then there exists an 
invariant point $q\in X_1\cap X_2$.
Let $U\subseteq X$ be an open invariant
affine subset which contains $q$.
Since $U$ has non-empty intersection 
with $T\cdot p_i$ it follows that
$p_i\in U$. 

The above analysis 
allows one to reduce to the case where 
both $p_1$ and $p_2$ are contained in 
an open affine subset $U\simeq
 k^n$ with the standard action
$(k^*)^n\times k^n\to
 k^n$ given by $(t_1,\dots,t_n)
\cdot (x_1,\dots,x_n) = (t_1x_1,\dots,t_nx_n)$.
Since $p_1$ and $p_2$ are distinct they
must differ for at least one coordinate, say the first one.
Choosing $v = (0,1,\dots,1)$
one has $t^v\cdot (x_1,x_2,\dots,x_n)
= (x_1,tx_2,\dots,tx_n)$ so that 
the two limits $q_1$ and $q_2$ remain distinct
and both lie on the invariant curve 
$\{x_2=\dots=x_n=0\}$.
\end{proof}

\begin{Definition}
Let $P\subseteq M_{\mathbb Q}$ be a lattice
polytope. The {\em length} of $P$ is
\[
 \ell(P)
 :=
 \min\{|L\cap M|-1\, | \, L
 \text{ is a 1-dimensional face of }P\}.
\]
\end{Definition}

\begin{Proposition}\label{TC2}
Let $P\subseteq M_{\mathbb Q}$ be a full 
dimensional lattice polytope such that 
the corresponding projective toric variety
$X_P\subseteq\mathbb P^{|P\cap M|-1}$
is smooth and linearly normal. Then the following are equivalent:
\begin{itemize}
\item[(\textit{a})] $\ell(P)\geq 3$;
\item[(\textit{b})] the $2$-Terracini locus of $X_P$ is empty;
\item[(\textit{c})] $X_P$ does not contain conics.
\end{itemize}
\end{Proposition}
\begin{proof}
The implication $(b)\Rightarrow (c)$ is trivial
since tangent lines to a conic intersect.
The implication $(c)\Rightarrow (a)$ follows 
from the fact that if $X_P$ does not contain 
conics in particular it does not contain invariant conics. 
Thus any edge of $P$ has length at least three.

We now prove $(a)\Rightarrow (b)$.
Thanks to Lemma \ref{spec} it is enough to prove that $T_{x_1}X_P\cap T_{x_2}X_P = \emptyset$ for two distinct points $x_1,x_2\in X_P$ lying on an invariant curve. 
If the two points are invariant ones, we claim
that the corresponding double points impose
independent conditions on the linear system
of hyperplane sections since imposing
each invariant double point means removing from $P$ a vertex together with 
all the points at distance one from it.

Since any two vertexes have distance at least three, it follows that two vertexes can not share the 
same point at distance one, which proves the claim.

If at least one point is not invariant then both points are contained in a common invariant affine chart.
Consider the local parametrization (\ref{par}) of $X_P$. We may prove the claim for the invariant curves in this chart since for the other ones it is enough to consider a change of variables in the torus. Up to a change of variables we may write (\ref{par}) as follows:
$$
\begin{array}{lccc}
\phi_P: & (k^*)^n & \longrightarrow & \mathbb{P}^{N}\\
& (u_1,\dots,u_n) & \mapsto & [1:u_1:\dots:u_1^{d_1}:\dots:u_n:\dots:u_n^{d_n},\phi_1:\dots:\phi_{N-d_1-\dots-d_n}]
\end{array}
$$
where $\phi_i(u_1,\dots,u_n)$ is a monomial depending on at least two of the $u_j$. Note that 
\begin{itemize}
\item[(i)] since $X_P$ is smooth the monomials $u_1,\dots,u_n$ must appear in the expression of $\phi_P$;
\item[(ii)] the monomials $u_1u_2,\dots,u_1u_n$ must also appear in the expression of $\phi_P$ since $X_P\subseteq\mathbb{P}^{N}$ is linearly normal and $u_1u_i$ is in the convex hull of $u_1^2$ and $u_i^2$.
\end{itemize}
Let $C\subset X_P$ be a torus invariant curve say $C = \phi_P(\{u_2 = \dots = u_n = 0\})$, and consider two distinct points $x_1 = (1,a,\dots,a^{d_1},0,\dots,0)$, $x_2 = (1,b,\dots,b^{d_1},0,\dots,0)$ of $C$. Now, since all the $1$-dimensional faces of $P$ have length at least three \text{(i)} and \text{(ii)} yield that the Terracini matrix $T_{X_P}(x_1,x_2)$ has a minor $M_{a,b}$ of size $(2n+2)$ of the following form 
$$
M_{a,b} = \left(\begin{array}{cc}
A_{a,b} & 0_{4,2n-2} \\ 
0_{2n-2,4} & B_{a,b}
\end{array}\right) 
$$
where $0_{i,j}$ denotes the $i\times j$ zero matrix
$$
A_{a,b} = \left(\begin{array}{cccc}
1 & a & a^2 & a^3\\ 
0 & 1 & 2a & 3a^2\\
1 & b & b^2 & b^3\\ 
0 & 1 & 2b & 3b^2
\end{array}\right)
\quad 
\text{and}
\quad
B_{a,b} = {\rm Diag}_{2n-2}(C_{a,b})
$$
is the size $2n-2$ matrix having $n-1$ copies of 
$$
C_{a,b} = \left(\begin{array}{cc}
1 & a\\
1 & b
\end{array}\right)
$$
on the main diagonal and whose other entries are zero. Note that $\det(A_{a,b}) = (b-a)^4$ and $\det(B_{a,b}) = (b-a)^{n-1}$. Therefore, $\det(M_{a,b}) = \det(A_{a,b})\det(B_{a,b}) = (b-a)^{n+3}$ proving the claim. 
\end{proof}

\begin{Remark}
The closure of the image of the map
$$
\begin{array}{lccc}
\phi: & (k^*)^2 & \longrightarrow & \mathbb{P}^{9}\\
& (u_1,u_2) & \mapsto & [1:u_1:u_1^2:u_1^{3}:u_2:u_2^2:u_2^{3}:u_1u_2:u_1^2u_2:u_1u_2^2]
\end{array}
$$
is the degree three Veronese embedding $V_3^2\subset\mathbb{P}^9$. Proposition \ref{TC2} yields that $T_2(V_{3}^{2})$ is empty. Now, let $X_P\subset\mathbb{P}^8$ be the projection of $V_3^2$ from $[0:\dots :0:1:0:0]$, that is the closure of the image of the map $\phi_P$ obtained by removing $u_1u_2$ from the expression of $\phi$. The variety $X_P$ is smooth. 

For any pair of points of the form $x_1 = \phi_P(a,0)$, $x_2 = \phi_P(-a,0)$ we have $T_{x_1}X_P\cap T_{x_2}X_P \neq \emptyset$ and so $T_2(X_P)$ is non empty. Therefore, the assumption on the linear normality of $X_P$ in Proposition \ref{TC2} can not be dropped.
\end{Remark}

\begin{Example}
We show that $2$-Terracini loci can originate from curves of degree greater than two. The closure of the image of the map
$$
\begin{array}{lccc}
\phi: & k^2 & \longrightarrow & \mathbb{P}^{6}\\
 & (u_1,u_2) & \mapsto & [u_1^3 : u_1^2u_2^3 : u_1^2u_2^2 : u_1^2u_2 : u_1u_2^2 : u_1u_2 : u_2]
\end{array}
$$
defined by the following lattice polygon
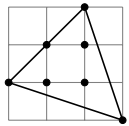
\begin{center}
\begin{tikzpicture}[scale=.5]
        \tkzInit[xmax=2,ymax=2,xmin=-1,ymin=-1]
        \tkzGrid
 \tkzDefPoint(1,2){P1}
 \tkzDefPoint(-1,0){P2}
 \tkzDefPoint(2,-1){P3}
 \tkzDefPoint(1,1){Q1}
 \tkzDefPoint(1,0){Q2}
 \tkzDefPoint(0,1){Q3}
 \tkzDefPoint(0,0){Q4}
 \tkzDrawSegments[color=black](P1,P2 P2,P3 P3,P1)
 \tkzDrawPoints[size=2.5](P1,P2,P3,Q1,Q2,Q3,Q4)
\end{tikzpicture}
\end{center}
is a surface $X\subset\mathbb{P}^6$ of degree eight, and $C = \overline{\phi(\{u_1-u_2 = 0 \})}\subset X$ is a rational normal quartic. The Terracini matrix of $\phi$ at any two general points 
of the curve $\{u_1-u_2 = 0 \}$ is the following
\[
\begin{pmatrix}
1 & 1 & 1 & 1 & 1 & 1 & 1 \\
3 & 2 & 2 & 2 & 1 & 1 & 0 \\
0 & 3 & 2 & 1 & 2 & 1 & 1 \\
t^3 & t^5 & t^4 & t^3 & t^3 & t^2 & t \\
3t^2 & 2t^4 & 2t^3 & 2t^2 & t^2 & t & 0 \\
0 & 3t^4 & 2t^3 & t^2 & 2t^2 & t & 1
\end{pmatrix}
\]
where we assume one of the two points to be the image of $(1,1)$. This matrix has always rank five
since the vector 
$$(t^5 - t^4 - 2t^3,  -t^4 + t^3,   -t^5 + t^4,   2t^2 + t - 1,   -t^3 + t^2,   -t^2 + t)$$
is in the kernel of its transpose.
\end{Example}

In what follows we make use of \cite[Theorem 3.5]{BC21} and \cite[Lemma 3.6]{BC21} to prove emptiness of Terracini loci for a class of embedded projective toric variety. In order to this we introduce an unbounded 
convex set attached to a projective variety.

\begin{Definition}
Let $X$ be a smooth projective variety. The \textit{ample body} of $X$ is
\[
 \mathscr A_X
 :=
 \{
  D\in {\rm N}^1(X)_{\mathbb R}\, :\, 
  D\cdot C\geq 1\text{ for any curve $C$}
 \}.
\]
\end{Definition}
Observe that $\mathscr A_X$ is an unbounded convex set such that $\mathscr A_X + {\rm Nef}(X) = \mathscr A_X$. 

\begin{Proposition}\label{polyhedral}
Let $X$ be a smooth projective variety
whose monoid of classes of curves is 
finitely generated.
Then $\mathscr A_X$ is a polyhedron
which is the Minkowski sum 
$$
 \mathscr A_X = A_X + {\rm Nef}(X)
$$
of a rational polytope $A_X\subseteq {\rm N^1}(X)_{\mathbb R}$
together with its recession cone ${\rm Nef}(X)
\subseteq {\rm N^1}(X)_{\mathbb R}$.
\end{Proposition}
\begin{proof}
By hypothesis there exist finitely many
irreducible curves $C_1,\dots,C_r$ whose 
classes in ${\rm N}_1(X)_{\mathbb R}$
form a Hilbert basis of the Mori cone of $X$. Any curve of $X$ is rationally equivalent to 
a non-negative sum of these curves.

It follows that $\mathscr A_X$ is intersection 
of the finitely many half-spaces 
$\{D\in {\rm N}^1(X)_{\mathbb R}\, :\, D\cdot C_i\geq 1\}$,
which proves the statement. Being $\mathscr A_X$ a polyhedron it is
Minkowski sum of a polytope $A_X$ 
together with its recession cone $\sigma_X$.
From the definition of $\mathscr A_X$ it 
immediately follows that ${\rm Nef}(X) = \sigma_X$.
\end{proof}

\begin{Proposition}\label{P_Body}
Let $X$ be a smooth projective variety
whose Mori cone is rational polyhedral.
Then $A_X$ is a normal lattice polytope if and only if for any $h$ the equality
$$(h\cdot\mathscr A_X) \cap {\rm N}^1(X) = \sum_{i=1}^h(\mathscr A_X\cap {\rm N}^1(X))$$ 
holds.
\end{Proposition}
\begin{proof}
By the definition of normal lattice polytope
we have that 
$$(h\cdot A_X) \cap {\rm N}^1(X) = \sum_{i=1}^h(A_X\cap {\rm N}^1(X))$$
and this is exactly what we need for the second condition to be satisfied.
\end{proof}

\begin{Remark}
Observe that if $A_X$ is a point then 
is must necessarily be a lattice point
and of course a polytope consisting 
of just one lattice point is normal.
\end{Remark}

\begin{Lemma}
\label{Lindep}
Let $X$ be a projective variety with smooth locus
$X^{o}$ and let $A,B$ be Weil divisors on $X$. 
Assume that for any subsets $S_A,S_B\subseteq
X^{o}$ of cardinality $n := |S_A|>1$ and 
$|S_B|=2$ one has
$$
H^1(X,\mathcal O_X(A)\otimes\mathcal I_{S_A}) = H^1(X,\mathcal O_X(B)\otimes\mathcal I_{S_B}) = 0.
$$
Then $H^1(X,\mathcal O_X(A+B)\otimes\mathcal I_S) = 0$ for any subset $S\subseteq X^{o}$ of cardinality $n+1$.
\end{Lemma}
\begin{proof}
Let $S\subseteq X^{o}$ be a subset of cardinality $n+1$.
Given any $p\in S$ write $S = U\cup\{p,q\}$ with $|U| = n-1$ non empty and
$p\neq q$. By hypothesis there exist 
$$
f\in H^0(X,\mathcal O_X(A)\otimes\mathcal I_U) \text{ and } g\in H^0(X,\mathcal O_X(B)\otimes\mathcal I_p)
$$
which do not vanish at $q$. Thus $fg\in H^0(X,\mathcal O_X(A+B)\otimes\mathcal I_{U\cup\{p\}})$ does not vanish at $q$.
\end{proof}

\begin{Corollary}\label{Cindep}
Let $X$ be a projective variety with smooth locus
$X^{o}$ and let $D$ be a Weil divisor which is sum 
of $n$ very ample divisors. 
Then any subset of $X^{o}$
of cardinality $n+1$ imposes independent conditions on $D$.
\end{Corollary}
\begin{proof}
Observe that if $B$ is very ample then it satisfies
the hypothesis on the divisor $B$ in Lemma~\ref{Lindep}.
The statement follows by induction on $n$.
\end{proof}

\begin{Proposition}\label{sevpts}
Let $X\subseteq\mathbb P^N$ be a projective 
variety embedded by a complete
linear system $|L|$. Assume that $L$ is sum of 
$2h-1\geq 3$ very ample divisors.
Then $T_h(X)$ is empty.
\end{Proposition}
\begin{proof}
Let $S\subseteq X^{o}$ be a subset of $h$ distinct
points. Write $L = A+B+C$ where both $A$ and $B$
are sum of $h-1$ very ample divisors and $C$ is very ample.
Since $C$ is very ample for any $p\in S$,
there is a smooth element in the linear system of $C$ through $p$
which does not contain any point of $S\setminus\{p\}$.

By Corollary~\ref{Cindep} $S$ imposes independent conditions on $A$. Since also $B+C$ is very ample the hypotheses 
of~\cite[Lemma 3.6]{BC21}
are satisfied so that for any $p\in S$ the scheme $(S\setminus\{p\})
\cup\{2p\}$ imposes independent conditions on
$B+C$. Finally, to conclude it is enough to apply \cite[Theorem 3.5]{BC21} with $L_1=B+C$ and $L_2=A$.
\end{proof}

Now, we are ready to prove the main result of this section.

\begin{thm}\label{main_to}
Let $P\subseteq M_{\mathbb Q}$ be a full 
dimensional lattice polytope such that 
the corresponding projective toric variety $X_P\subseteq\mathbb P^{|P\cap M|-1}$, embedded by a complete linear system $|L|$, is smooth. If $A_{X_P}$ is
a normal lattice polytope then the following 
are equivalent:
\begin{itemize}
\item[(\textit{a})] $\ell(P) = s$ with $s\geq 2h-1$;
\item[(\textit{b})] $L\in s\cdot\mathscr A_X$;
\item[(\textit{c})] $T_h(X)$ is empty.
\end{itemize}
\end{thm}
\begin{proof}
Note that $(b)\Rightarrow (c)$ follows from Propositions \ref{P_Body} and \ref{sevpts}. Now, assume that $s< 2h-1$. Then $X_P$ contains an invariant curve embedded as a curve $C\subset \mathbb P^{|P\cap M|-1}$ of degree $s$. Since $X_P$ is smooth, arguing as in the proof of Proposition \ref{TC2}, we get that $C\subset \mathbb P^{|P\cap M|-1}$ is a rational normal curve of degree $s$. Since $s< 2h-1$ the span of any $h$ tangent lines of $C$ has dimension smaller than expected and hence the span of $h$ tangent spaces of $X_P$ at points of $C$ has dimension smaller than expected as well. Therefore, $T_h(X)$ is non empty proving that $(c)\Rightarrow (a)$.

Finally, to prove that $(a)\Rightarrow (c)$ note that since $\ell(P) = s$ all the toric invariant curves in $X_P$ have degree at least $s$, and since these curves generate the Mori cone of $X_P$ we get that $L\in s\cdot\mathscr A_X$.
\end{proof}

There exist smooth toric varieties for which $A_X$ is not a normal lattice polytope as the following example shows.

\begin{Example}
Let $X$ be the smooth toric surface whose fan is the following:
 
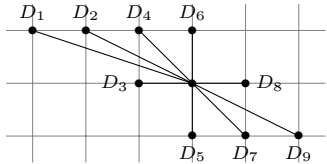
\begin{center}
  \begin{tikzpicture}[scale=.7]
  \draw[help lines] (-3.5,-1.5) grid (2.5,1.5);
  \foreach \x/\y in {0/-1, 2/-1, 1/0, 1/-1, -1/1, 0/1, -2/1, -3/1, -1/0}
  {
   \draw (0,0) -- (\x,\y);
   \node at (\x,\y) {\tiny $\bullet$};
  }
  \node[above] at (-3,1) {\tiny $D_1$};
  \node[above] at (-2,1) {\tiny $D_2$};
  \node[left] at (-1,0) {\tiny $D_3$};
  \node[above] at (-1,1) {\tiny $D_4$};
  \node[below] at (0,-1) {\tiny $D_5$};
  \node[above] at (0,1) {\tiny $D_6$};
  \node[below] at (1,-1) {\tiny $D_7$};
  \node[right] at (1,0) {\tiny $D_8$};
  \node[below] at (2,-1) {\tiny $D_9$};
  \node at (0,0) {\tiny $\bullet$};
  \end{tikzpicture}
 \end{center}
Then $A_X$ is not a lattice polytope. To prove this observe that the $\mathbb Q$-divisor
$$D = D_3+D_4+4D_5+3D_6+4D_7+\frac{7}{2}D_8+5D_9$$
has intersection product at least one with any $D_i$ and the equality holds for $i\in\{1,2,3,4,5,7,8\}$.
 Since the classes of these seven divisors form a 
 basis of the rational Picard lattice, it follows that the class of
 $D$ is a vertex of $A_X$. On the other hand
 $D\cdot D_6 = \frac{3}{2}$, so that the class of $D$ 
 can not be in the integral Picard group and so $A_X$ con not be a lattice polytope.
\end{Example}

\section{Applications}\label{appli}
In this section we apply our main results in Sections \ref{TSIB} and \ref{TVAB} to several classes of projective varieties.

\begin{Proposition}\label{Pic2}
Let $X\subset\mathbb{P}^N$ be a smooth toric variety of Picard rank two embedded by the complete linear system of a divisor $L\in (2h-1)\cdot\mathscr{A}_X$. Then $T_h(X)$ is empty.
\end{Proposition}
\begin{proof}
If $X$ is a smooth projective toric variety
with Picard rank two, then its nef cone is
generated by two primitive rays, so that
it is simplicial. 

To prove that it is also smooth
we proceed as follows. Let $D_1,\dots,D_r$
be the prime torus invariant divisors of $X$
and say that $D_1, D_2$ are the two 
whose classes generate ${\rm Nef}(X)$.
The intersection $\bigcap_{i=3}^r D_i$ is 
a torus invariant point $p\in X$ and the 
local divisor class group ${\rm Cl}(X,p)$
is generated by the class of the prime 
invariant divisors which contain $p$,
that is by the $D_i$ with $i\geq 3$.
Thus ${\rm Cl}(X,p) = {\rm Cl}(X\setminus
D_1\cup D_2) = {\rm Cl}(X)/\langle D_1,D_2\rangle$.

On the other hand ${\rm Cl}(X,p)$ is trivial,
being $X$ smooth, so that we can conclude
that ${\rm Cl}(X) = \langle D_1,D_2\rangle$,
which proves that the classes of $D_1$ and 
$D_2$ form a basis of ${\rm Cl}(X)$.
This shows that ${\rm Nef}(X)$ is smooth.
Being this cone smooth and simplicial we get
that $A_X$ is a point. Finally, to conclude it is enough to apply Theorem \ref{main_to}. 
\end{proof}

Choose positive integers $a_1\leq a_2\leq \dots\leq a_n$ such that $\sum_{i=1}^na_i = N-n+1$, $\Lambda_i\cong\mathbb{P}^{a_i}\subset\mathbb{P}^N$ complementary linear subspaces, $C_i\subset\mathbb{P}^{a_i}$ rational normal curves, isomorphisms $\phi_i:C_1\rightarrow C_i$, and consider the rational normal scroll $S_{a_1,\dots,a_n}=\bigcup_{p\in C_1}\left\langle p,\phi_2(p),\dots,\phi_n(p)\right\rangle$. Let $H$ be the restriction to $S_{a_1,\dots,a_n}$ of the hyperplane section of $\mathbb{P}^N$, $D = dH$ and $S_{a_1,\dots,a_n,d}\subset\mathbb{P}^{N_d}$ the image of $S_{a_1,\dots,a_n}$ via the embedding induced by $D$.

\begin{Corollary}\label{scroll}
If $h\leq\lceil\frac{d}{2}\rceil$ then $T_h(S_{a_1,\dots,a_n,d})$ is empty.
\end{Corollary}
\begin{proof}
Since $S_{a_1,\dots,a_n,d}$ is a smooth toric variety of Picard rank two the claim follows from Proposition \ref{Pic2}.
\end{proof}

Let $\pmb{n}=(n_1,\dots,n_r)$ and $\pmb{d} = (d_1,\dots,d_r)$ be two
$r$-uples of positive integers, with $n_1\leq \dots \leq n_r$ and
$N(\pmb{n},\pmb{d})=\prod_{i=1}^r\binom{n_i+d_i}{n_i}-1$. The
Segre-Veronese variety $SV_{\pmb{d}}^{\pmb{n}}$ is the image in
$\mathbb{P}^{N(\pmb{n},\pmb{d})}$ of
$\mathbb{P}^{n_1}\times\dots\times\mathbb{P}^{n_r}$ via the embedding
induced by $ 
\mathcal{O}_{\mathbb{P}^{\pmb{n}} }(d_1,\dots,
d_r)=\mathcal{O}_{\mathbb{P}(V_1^{*})}(d_1)\boxtimes\dots\boxtimes
\mathcal{O}_{\mathbb{P}(V_1^{*})}(d_r)$.

As a consequence of Theorem \ref{sevpts} we recover the results on Terracini loci of Segre-Veronese varieties in \cite[Theorem 1.3]{Bal22}.

\begin{Corollary}\label{TCSV}
If $h\leq\lceil\frac{d_i}{2}\rceil$ for all $i = 1,\dots,r$ then $T_h(SV_{\pmb{d}}^{\pmb{n}})$ is empty.
\end{Corollary}
\begin{proof}
Since $h\leq\lceil\frac{d_i}{2}\rceil$ for all $i = 1,\dots,r$ we may write $L = 2A+B$ where $L = \mathcal{O}_{\mathbb{P}^{\pmb{n}} }(d_1,\dots,d_r)$, $A = \mathcal{O}_{\mathbb{P}^{\pmb{n}} }(h-1,\dots,h-1)$ and $B = \mathcal{O}_{\mathbb{P}^{\pmb{n}} }(b_1,\dots,b_r)$ with $b_i \geq 1$ for all $i = 1,\dots,r$.

Since the Mori cone of $\mathbb{P}^{n_1}\times\dots\times\mathbb{P}^{n_r}$ is generated by the classes $[l_i]$, where $l_i\subset\mathbb{P}^{n_i}$ is a line to conclude it is enough to apply Theorem \ref{sevpts}.
\end{proof}

\begin{Corollary}\label{TCSV2}
If $2h\leq\lceil\frac{d_i}{2}\rceil$ for all $i = 1,\dots,r$ then any point of $\sec_h(SV_{\pmb{d}}^{\pmb{n}})\setminus \sec_{h-1}(SV_{\pmb{d}}^{\pmb{n}})$ is $h$-identifiable and $\sec_h(SV_{\pmb{d}}^{\pmb{n}})$ is smooth outside of $\sec_{h-1}(SV_{\pmb{d}}^{\pmb{n}})$.
\end{Corollary}
\begin{proof}
The claim follows immediately from Corollary \ref{TCSV} and Proposition \ref{h-2h}.
\end{proof}

\begin{Proposition}\label{hom}
Let $X\subset\mathbb{P}^N$ be a homogeneous variety embedded by the complete linear system of a divisor $L\in (2h-1)\cdot\mathscr{A}_X$. Then $T_h(X)$ is empty.
\end{Proposition}
\begin{proof}
If $X$ is a homogeneous variety then ${\rm Nef}(X) = {\rm Eff}(X)$ and both 
cones are smooth and simplicial~\cite[Proposition 1.4.1]{Bri05}. In particular $A_X$ is the lattice point 
given by the sum of the primitive generators of the rays of the nef cone. Hence, to conclude it is enough to apply Propositions \ref{P_Body} and \ref{sevpts}.
\end{proof}

If $L\in r\cdot\mathscr{A}_X$ with $r\leq 2h-2$ the Terracini locus is in general non empty as the following result shows.

\begin{Proposition}\label{grass}
Let $\mathbb{G}(r,n)\subset \mathbb{P}^N$ be the Grassmannian of $r$-linear spaces in $\mathbb{P}^n$ embedded with the Pl\"ucker embedding. Then $T_2(\mathbb{G}(r,n))$ is the closure of
$$\{([U],[U'])\in
\mathbb{G}(r,n)\times \mathbb{G}(r,n)\: | \: U\neq U' \text{ and } \dim (U\cap U')\geq r-2 \}/ S_2$$
in $(\mathbb{G}(r,n)\times \mathbb{G}(r,n))/S_2$.
\end{Proposition}
\begin{proof}
By \cite[Lemma 6.5]{MR18} we have the following characterization of the tangent space of the Grassmannian at a point
$[U]\in \mathbb{G}(r,n)$:
$$T_{[U]}(\mathbb{G}(r,n))
=\langle
e_I \:|\: d(I,\{0,\dots,r\})\leq 1
\rangle
$$
where 
$(e_0,\dots,e_n)$ is a basis of $k^{n+1},$ 
$U$ is generated by $e_0,\dots,e_r\in \mathbb{P}^n,$
$e_I = e_{i_0}\wedge \cdots \wedge e_{i_r}$
and $d(I,J)$ is the Hamming distance between the two lists $I$ and $J.$

Now, given $[U],[V]\in \mathbb{G}(r,n)$ let $s:=\dim(U \cap V)$, after a base change, we can write
$U=\langle e_0,\dots,e_r \rangle$
and $U'=\langle e_0,\dots,e_s,e_{r+1},\dots,e_{r+(r-s)} \rangle.$
Note that $s=\dim (U\cap U').$ Therefore
$$
\begin{small}
\begin{array}{ll}
T_{[U]}(\mathbb{G}(r,n)) \cap  T_{[U']}(\mathbb{G}(r,n)) & = \langle e_I \: | \: d(I,\{0,\dots,r\})\leq 1\rangle\cap\langle
e_I \: | \: d(I,\{0,\dots,s,r+1,\dots,r+(r-s)\})\leq 1\rangle\\ 
 & = \langle e_I \: | \: d(I,\{0,\dots,r\})\leq 1 \mbox{ and } d(I,\{0,\dots,s,r+1,\dots,r+(r-s)\})\leq 1
\rangle.
\end{array} 
\end{small}
$$
Hence, $T_{[U]}(\mathbb{G}(r,n)) \cap  T_{[U']}(\mathbb{G}(r,n))$ is empty if and only if $d(\{0,\dots,r\},
\{0,\dots,s,r+1,\dots,r+(r-s)\}) \geq 3$
which is equivalent to $s\leq r-3$.
\end{proof}

\stepcounter{thm}
\subsection{Fano varieties}\label{Fano}
We show that Theorem \ref{main_to} applies to smooth toric Fano varieties of dimension at most four.
\begin{Proposition}\label{P_Fano}
Let $X$ be a smooth toric Fano variety of dimension at most four. Then $A_X$ is a lattice point.
\end{Proposition}
\begin{proof}
We load the $n$-th entry from the smooth toric Fano varieties database
using the function {\tt FanoX} from the following
library: \url{https://github.com/alaface/Terracini_Loci}.
The following function computes the 
ample polytope of the toric variety $X$.
\begin{tcolorbox}
\begin{verbatim}
AX := function(X)
 forms := IntersectionForms(X);
 pol := &meet[HalfspaceToPolyhedron(v,1) : v in forms];
 return CompactPart(pol);
end function;
\end{verbatim}
\end{tcolorbox}

Smooth toric varieties of dimension at most four are the first $147$ entries of the database.
For each such variety we check with the following script
\begin{tcolorbox}
\begin{verbatim}
> time {#Points(AX(FanoX(n))) : n in [1..147]};
{ 1 }
Time: 153.710
\end{verbatim}
\end{tcolorbox}
that the ample polytope consists of a point.
\end{proof}

However, there are smooth toric Fano $5$-fold for which $A_X$ is not a point. 

\begin{Example}
Let $X\subseteq\mathbb P^{192}$ be the 
smooth projective toric Fano $5$-fold, ID. $556$
in the  \href{http://www.grdb.co.uk/search/toricsmooth}
{Graded Ring Database}, whose
Cox ring $\mathbb C[x_1,\dots,x_{10}]$ has 
the following grading matrix
\[
 \begin{bmatrix}
    0&0&0&1&1&0&0&0&0&1\\
    0&0&1&0&0&0&1&1&0&0\\
    0&0&1&0&1&0&0&1&1&0\\
    0&1&1&0&1&0&0&1&0&1\\
    1&0&0&0&0&1&0&0&1&0 
 \end{bmatrix}
\]
Let $D_i$ be the $i$-th prime invariant
torus divisor of $X$. Then $A_X$ is a lattice 
segment of length one whose lattice points
are the classes of $-K_X$ and $-K_X+D_3$.
\end{Example}

\begin{proof}[Proof of Theorem \ref{thC}]
Since by Proposition \ref{P_Fano} $A_X$ is a lattice point Theorem \ref{main_to} applies.
\end{proof}

\bibliographystyle{amsalpha}
\bibliography{Biblio}
\end{document}